\documentclass[reqno]{amsart}
\usepackage{amsmath,amsthm,amscd,amssymb,amsfonts, amsbsy}
\usepackage[usenames, dvipsnames]{color}
\usepackage{pxfonts}
\usepackage{enumerate}
\usepackage{latexsym}
\usepackage{mathrsfs}
\usepackage{xparse}
\usepackage{marginnote}
\usepackage{todonotes}
\usepackage{hyperref}

\numberwithin{equation}{section}

\usepackage{verbatim}

\allowdisplaybreaks[4]
\usepackage{todonotes}

\theoremstyle{plain}
\newtheorem{theorem}{Theorem}[section]
\newtheorem{lemma}[theorem]{Lemma}
\newtheorem{corollary}[theorem]{Corollary}
\newtheorem{proposition}[theorem]{Proposition}

\theoremstyle{definition}
\newtheorem{definition}[theorem]{Definition}
\newtheorem{condition}[theorem]{Condition}

\theoremstyle{remark}
\newtheorem{remark}[theorem]{Remark}

\DeclareMathOperator{\osc}{osc}

\DeclareMathOperator{\loc}{loc}
\DeclareMathOperator{\dv}{div}

\providecommand{\set}[1]{\{#1\}}

\providecommand{\bigset}[1]{\bigl\{#1\bigr\}}

\providecommand{\biggset}[1]{\biggl\{#1\biggr\}}
\providecommand{\Biggset}[1]{\Biggl\{#1\Biggr\}}

\providecommand{\abs}[1]{\lvert#1\rvert}

\providecommand{\bigabs}[1]{\bigl\lvert#1\bigr\rvert}

\providecommand{\ip}[1]{\langle#1\rangle}

\providecommand{\norm}[1]{\lVert#1\rVert}

\newcommand{\bR}{\mathbb{R}}

\begin{document}

\title{Gradient continuity for $p$-Laplacian obstacle problems under mean oscillation conditions}

\author[S. Lee]{Sungjin Lee}
\address[S. Lee]{Department of Mathematics, Sogang University, 35 Baekbeom-ro, Mapo-gu, Seoul 04107, Republic of Korea}
\email{sungjinlee@sogang.ac.kr}
\thanks{S. Lee was supported by the BK21 FOUR (Fostering Outstanding Universities for Research) funded by the Ministry of Education (MOE, Korea) and National Research Foundation of Korea (NRF) with grant No.4120240415042, and by the National Research Foundation of Korea(NRF) grant funded by the Korea government(MSIT) (RS-2025-24533680).}

\author[J. Ok]{Jihoon Ok}
\address[J. Ok]{Department of Mathematics, Sogang University, 35 Baekbeom-ro, Mapo-gu, Seoul 04107, Republic of Korea}
\email{jihoonok@sogang.ac.kr}
\thanks{J. Ok was supported by the National Research Foundation of Korea(NRF) grant funded by the Korea government(MSIT) (NRF-2022R1C1C1004523).}

\subjclass[2020]{Primary 35J92 ; Secondary 35R35, 35B65, 49J40}
\keywords{$p$-Laplacian; Obstacle problem; Dini Mean Oscillation; $C^1$-regularity; variational inequality }

\begin{abstract}
We establish the $C^1$-regularity of solutions to the obstacle problems associated with $p$-Laplacian type equations, where $1<p<\infty$. Specifically, we prove that the gradient of the solution is continuous under a Dini mean oscillation ($\mathsf{DMO}$) type condition on the data, which includes the coefficient matrix, the source term, and the obstacle function. This result relaxes the classical Dini continuity assumption on the data to a more general mean oscillation condition.
\end{abstract}

\maketitle

\section{Introduction and Main Result} \label{sec1}

We study the regularity of the gradient of the solution to the obstacle problem associated with the following $p$-Laplacian equation for $1<p<\infty$:
\begin{equation} \label{main_eq}
-\dv ( \ip{A(x)Du, Du }^{\frac{p-2}{2}} A(x) Du)=-\dv F \quad\text{in }\ \Omega,
\end{equation}
where $\Omega \subset \mathbb{R}^n$ ($n\ge 2$) is an open set, $A:\Omega\to M_n(\mathbb{R})$ is a symmetric matrix-valued function satisfying the ellipticity condition; that is, there exists $\Lambda>1$ such that
\begin{equation} \label{elliptic}
\begin{split}
\Lambda^{-1}\abs{\xi}^2\leq \ip{A(x)\xi,\xi}&\leq \Lambda\abs{\xi}^2,\quad\forall x\in \Omega\ \ \text{and}\ \ \forall\xi\in\bR^n,
\end{split}
\end{equation}
and $F\in L^p(\Omega)$.
Given an obstacle function $\psi \in W^{1,p}(\Omega)$, we define the admissible set
\[
\mathcal{A}_\psi:=\set{f\in W^{1,p}(\Omega)~:~f\geq\psi~\text{a.e. in $\Omega$}}.
\]

We say that $u\in \mathcal{A}_\psi$ is a solution to the variational inequality
\begin{align} \label{main_vari ineq}
\int_\Omega \ip{A(x) Du,Du}^{\frac{p-2}{2}} A(x)Du\cdot D(\varphi-u)  \,dx \geq \int_\Omega F\cdot D(\varphi-u) \,dx
\end{align}
for all $\varphi\in\mathcal{A}_\psi$ such that $\varphi-u$ has compact support in $\Omega$. We note that $u\in \mathcal{A}_\psi$ solves the variational inequality \eqref{main_vari ineq} if and only if it is a minimizer of the energy functional
\[
w\in \mathcal{A}_\psi~\mapsto~ \int_\Omega\big( \ip{A(x)Dw,Dw}^{\frac{p}{2}}-pF\cdot Dw \big) \,dx.
\]

Regularity theory for the gradient of the solution to the variational inequality \eqref{main_vari ineq} has been actively studied. We refer to, for instance, \cite{Choe91,ChoeLewis91,Fuchs90}, \cite{BDM11,BCW12}, and \cite{Ok17} for results on H\"older continuity, $L^q$-integrability, and continuity, respectively, as well as the references therein.

The continuity of the gradient of solutions, known as $C^1$-regularity, is regarded as a borderline regularity between H\"older continuity and $L^q$-integrability. In the linear case ($p=2$) without an obstacle, i.e.,
 \begin{equation}\label{eq_linear}
 -\dv (A(x) Du)= -\dv F,
 \end{equation}
it is a classical and well-known result that if $A$ and $F$ are Dini continuous (see Definition~\ref{def:Dini}), then the gradient of weak solution $Du$ is continuous; see, e.g., \cite{HW55,Burch78}. This result is sharp, as demonstrated by \cite{JMV09}, who constructed a counterexample showing that $Du$ is not necessarily locally bounded even when $A$ is continuous.

For $p$-Laplacian type problems, Kuusi and Mingione \cite{KM14} considered the following $p$-Laplace system with coefficients:
$$
-\,\mathrm{div} \left(a(x)|Du|^{p-2}Du\right) = f
\quad \text{in }\ \Omega, 
\quad u, f:\Omega \to \mathbb{R}^N,
$$
and showed that if the coefficient function $a$ is Dini continuous and $f$ belongs to the Lorentz space $L^{n,1}$, then the gradient $Du$ is continuous. We also refer to \cite{DM10,KM13,Mingione11} for pointwise potential estimates for the gradient of weak solutions in the scalar case $N=1$, which imply $C^1$-regularity results. Furthermore, we refer to \cite{KM12} for the local boundedness of the gradient in parabolic $p$-Laplace equations with divergence data as in \eqref{main_eq}, to \cite{Baroni23,Baroni25,Ok15} for nonlinear problems with nonstandard growth, and to \cite{BMK14,Ok17-1} for variational problems. Finally, regarding the variational inequality \eqref{main_vari ineq} with $F=|G|^{p-2}G$, the second author \cite{Ok17} proved that if $G$ and $D\psi$ are $\theta_p$-Dini continuous with $\theta_p:= \min\{p-1,p'/2\}\in (0,1]$ (see Definition~\ref{def:Dini}), where $p':=\frac{p}{p-1}$, then $Du$ is continuous.

In $C^1$-regularity theory for divergence type problems, it has been a very interesting question whether the Dini condition for the modulus of continuity of the data can be replaced by a condition on their mean oscillation. In the linear case \eqref{eq_linear}, Dong and Kim \cite{DK17} showed that if $A$ and $F$ satisfy the $\mathsf{DMO}$ condition (see Definition~\ref{def:Dini}), then $Du$ is continuous. See also \cite{DEK18, YanyanLi17, DEK21, JMV09,  DK21} for related results. Recently, H\"ast\"o, Lee, and the second author of the present paper \cite{HLO25} considered the following very general class of nonlinear homogeneous equations of the form
\[
\dv \mathbf{a}(x,Du) =0,
\]
which includes $p$-Laplace equations with coefficients, $p(x)$-Laplace equations, and double phase equations, and showed gradient continuity under a Dini condition for a certain mean oscillation of the nonlinearity $\mathbf{a}(x,\xi)$. Specifically, if $\mathbf{a}(x,\xi)=a(x)|\xi|^{p-2}\xi$, the condition on $\mathbf{a}(x,\xi)$ in \cite{HLO25} means that $\omega^{(q)}_a$ (see \eqref{eq:Lpmeanoscillation})  for some $q>2$ satisfies the Dini condition.

In this paper, we focus on variational inequalities, i.e., obstacle problems. Instead, we consider a rather simpler structure: the $p$-Laplacian with a coefficient matrix $A$, and obtain the following $C^1$-regularity result under a mean oscillation type condition on the $\mathsf{data}$: coefficients, source term, and obstacle. 

\begin{theorem} \label{main_theorem}
Let $1<p<\infty$, $A$ be a symmetric $n\times n$ matrix-valued function on $\Omega$ satisfying \eqref{elliptic}, $F\in L^p(\Omega)$, $\psi\in W^{1,p}(\Omega)$, and $u \in W^{1,p}(\Omega)$ be a solution to the variational inequality \eqref{main_vari ineq}. Suppose $A$, $F$, and $\psi$ satisfy the Dini mean oscillation condition (Condition~\ref{datacondi}). Then $Du$ is continuous in $\Omega$.
\end{theorem}

The above theorem still holds true for the non-obstacle problem case; that is, it applies to the $p$-Laplace equation with coefficients and source term \eqref{main_eq} by removing $\psi$. 

We now highlight the novelty of our approach and the mean oscillation condition. 
In the linear case in \cite{DK17}, the $\mathsf{data}$ are assumed to be $\mathsf{DMO}$, which is weaker than Condition~\ref{datacondi}. The proof in the linear case is based on weak type-$(1,1)$ estimates for linear equations—which make it possible to work directly with the $\mathsf{DMO}$ condition of $A$ and $F$—and a priori estimates along with the method of continuity.
In contrast, for our problems \eqref{main_vari ineq} or \eqref{main_eq} with $p\neq 2$, these techniques are no longer available. Instead, we use known Calder\'on--Zygmund type $L^q$-estimates for any $q>p$ and adapt the sharp $L^1$-comparison estimates and iteration argument established in \cite{HLO25} to the variational inequality problems, which are simpler than the ones in \cite{KM14,Ok17}. In this regard, $\mathsf{DMO}_{q}$ for some $q>2$ seems to be a quite natural mean oscillation condition corresponding to the $C^1$-regularity for $p$-Laplacian problems with $p\neq 2$.

Our condition on the $\mathsf{data}$ generalizes that of \cite{Ok17} to a mean oscillation version. Moreover, regarding the regularity of the source term, we also extend the result in \cite{KM14} to a more general source term satisfying a $\mathsf{DMO}$ type condition when $1<p\le 2$.  We will discuss on the relationships between regularity conditions for $C^1$-regularity in Remark~\ref{rmk:Dinimeancond}. We would like to point out that when $p>2$, Condition~\ref{datacondi} implies that $|D\psi|^{p-2}D\psi$ and $F$ are $p'/2$-$\mathsf{DMO}_{p'}$, where $0<p'/2<1$. This is slightly stronger than our expected condition, $\mathsf{DMO}_{q}$ for some $q\geq 1$. Hence, this may be improved in the future by applying a more delicate comparison and iteration argument as in \cite{KM14}.

The paper is organized as follows. In Section \ref{sec2}, we introduce regularity assumptions on the $\mathsf{data}$ and some preliminary results.  In Section \ref{sec3}, we derive comparison estimates. In Section \ref{sec4}, we prove local boundedness of $Du$ and Theorem \ref{main_theorem}.

\section{Preliminaries}  \label{sec2}

In this section, we describe the regularity of the $\mathsf{data}$-condition, more precisely, and also present basic inequalities on the structure of the $p$-Laplacian equation \eqref{main_eq}, as well as several well-known regularity results.

\subsection{Dini Mean Oscillation}\label{subsec:Dini}
We begin by introducing Dini type conditions.
\begin{definition} \label{def:Dini}
Let $g: [0,R]\to [0,\infty)$ for some $R\in(0,1]$, $f: \Omega\to \mathbb{R}^N$ with $N\in \mathbb{N}$, $\theta\in(0,1]$, and $1\le q<\infty$.
\begin{itemize}
\item  We say that $g$ satisfies the \textit{Dini condition} if 
\[
\int_0^R g(t)\frac{dt}{t} <\infty.
\]  
\item  For $0<r \le \mathrm{diam}(\Omega)$, we denote
\[
\mathcal{O}_f(r):= \sup\{|f(x)-f(y)| : x,y\in\Omega \text{ with } |x-y|\le r\}.
\]
 We say that $f$ is $\theta$-\textit{Dini continuous} if the function $r \mapsto \mathcal{O}_f(r)^{\theta}$ satisfies the Dini condition.
 In particular, if $\theta=1$, we simply say that $f$ is Dini continuous.
 
 \item For $0<r \le \mathrm{diam}(\Omega)$ and $x\in \overline{\Omega}$, we denote
 \begin{equation}\label{eq:Lpmeanoscillation}
\omega^{(q)}_f(r, x):= \bigg(\fint_{\Omega_r(x)} \,\abs{f(y)-\bar{f}^{x,r}}^q\,dy\bigg)^{\frac{1}{q}}
\quad\text{and}\quad
\omega^{(q)}_f(r):= \sup_{x \in \overline{\Omega}} \omega^{(q)}_f(r, x),
\end{equation}
where $\Omega_r(x):= \Omega\cap B_r(x)$ and $\bar{f}^{x,r}:=\fint_{\Omega_r(x)} f$. 
 We say that $f$ is $\theta$-$\mathsf{DMO}_{q}$ (\textit{Dini Mean Oscillation}) if the function $r \mapsto \omega^{(q)}_f(r)^{\theta}$ satisfies the Dini condition.
 In particular, we say that $f$ is $\mathsf{DMO}_{q}$ if $\theta=1$, $\theta$-$\mathsf{DMO}$ if $q=1$, and simply $\mathsf{DMO}$ if $\theta=q=1$.
 \end{itemize}
\end{definition}

We provide a few remarks regarding the above definition.
\begin{remark} \label{rmk:Dini}
\begin{itemize}
\item[(i)] If $0<\theta_1< \theta_2\le 1$, then $\theta_1$-Dini continuity implies $\theta_2$-Dini continuity, and $\theta_1$-$\mathsf{DMO}_q$ implies $\theta_2$-$\mathsf{DMO}_q$. Dini continuity implies $\mathsf{DMO}_q$ for any $1\le q <\infty$. Note that the converse does not hold; see \cite{DK17} for an example.
\item[(ii)] $\theta$-$\mathsf{DMO}_{q}$ implies uniform continuity for any $q\geq 1 $ and $0<\theta\leq 1$; see the appendix of \cite{HK20}.
\item[(iii)] H\"older's inequality yields that $\mathsf{DMO}_{q_2}$ implies $\mathsf{DMO}_{q_1}$ if $1\le q_1 < q_2$. However, the converse is not known. This is in contrast to the $\mathsf{BMO}$ spaces, where the equivalence of the conditions for all $1 \le q < \infty$ follows from the John--Nirenberg inequality. The underlying reason for this difficulty is that the function $r \mapsto \omega_f^{(q)}(r)$ is not necessarily increasing.
\item[(iv)] Suppose $f\in W^{1,1}(\Omega)$ and $Du\in L^{n,1}(\Omega)$. Then $f$ is $\mathsf{DMO}_{q}$ for any $1\le q <\infty$. We refer to \cite{HLO25} for the proof.
\item[(v)] The class of $\theta$-$\mathsf{DMO}_{q}$ functions is not contained in the class of Dini continuous functions. For example, let $f(x)=(-\log \abs{x})^{-\alpha}$. Then $f$ is Dini continuous if and only if $\alpha>1$. However, $f$ is $\theta$-$\mathsf{DMO}_{q}$ for every $q\ge 1$ if $1/\theta-1<\alpha$. 
\item[(vi)] $f$ is $\mathsf{VMO}$ if and only if $\lim_{r\rightarrow 0} \omega_f^{(q)}(r)=0$ for any $q\geq 1$. Hence if $f$ is $\theta$-$\mathsf{DMO}_{q}$, then $f$ is $\mathsf{VMO}$.
\end{itemize}
\end{remark}

We now state the main condition for the $C^1$-regularity result.  

\begin{condition} \label{datacondi}
We assume that the $\mathsf{data}$ (coefficients, source term, and obstacle) satisfies the following Dini mean oscillation type condition: there exists $\kappa>0$ such that the mean oscillation function 
$\omega_{\mathsf{data}}:\bR_+\rightarrow\bR$ defined by
\[
\omega_{\mathsf{data}}(r):= \sup_{x\in \Omega} \omega_{\mathsf{data}}(r, x),
\]
where 
\begin{equation*}
\omega_{\mathsf{data}}(r,x):=
\begin{cases}
 \omega^{(2+\kappa)}_A(r,x )+\omega_{|D\psi|^{p-2}D\psi}^{(p')}(r,x)^{\frac{p'}{2}}+ \omega_F^{(p')}(r,x)^{\frac{p'}{2}}  &\text{if } 2< p <\infty,\\[1ex]
\omega^{(2+\kappa)}_A(r,x )+ \omega^{(2+\kappa)}_{|D\psi|^{p-2}D\psi}(r,x )+\omega^{(2+\kappa)}_F (r,x ) &\text{if } 1< p \le 2,
\end{cases}
\end{equation*}
satisfies the Dini condition, i.e.,
\begin{equation*}           
\int_0^1 \frac{\omega_{\mathsf{data}}(t)}{t} \,dt <+\infty.
\end{equation*} 
This condition means that $A$ is $\mathsf{DMO}_{2+\kappa}$; $|D\psi|^{p-2}D\psi$ and $F$ are $p'/2$-$\mathsf{DMO}_{p'}$ for $p > 2$, and $\mathsf{DMO}_{2+\kappa}$ for $1 < p \le 2$. 
\end{condition}

\begin{remark} \label{rmk:Dinimeancond}
We now clarify the relationship between Condition~\ref{datacondi} and the conditions on the $\mathsf{data}$ in \cite{HLO25,KM14,Ok17}. As mentioned in the introduction, the $C^1$-regularity result for the $p$-Laplacian case in \cite{HLO25} requires the coefficient function to be $\mathsf{DMO}_{q}$ for some $q>2$.  This corresponds exactly to the condition that the coefficient matrix $A$ is $\mathsf{DMO}_{2+\kappa}$ for some $\kappa>0$.

We note from \eqref{eq0127tue3} that 
\[\begin{split}
&\fint_{B_r(x)}\left||D\psi(y)|^{p-2}D\psi(y) -  |\overline{D\psi}^{x,r}|^{p-2}\overline{D\psi}^{x,r}\right|^q\,dy \\
&\quad \sim \fint_{B_r(x)}\left(|D\psi(y)|^2+  |\overline{D\psi}^{x,r}|^2\right)^{\frac{q(p-2)}{2}}|D\psi(y)-\overline{D\psi}^{x,r}|^q\,dy\\
&\quad \lesssim \begin{cases} 
\fint_{B_r(x)}|D\psi-\overline{D\psi}^{x,r}|^{(p-1)q}\,dy 
&\text{if } 1<p\le 2,\\[1ex]
\|D\psi\|_{L^\infty(B_r(x))}^{(p-2)q}\fint_{B_r(x)}|D\psi-\overline{D\psi}^{x,r}|^q\,dy
&\text{if } p> 2.
\end{cases}
\end{split}\]
Therefore, we see that if $D\psi$ is $\theta_p$-Dini continuous with $\theta_p=\min\{p-1,p'/2\}$, then $|D\psi|^{p-2}D\psi$ is $p'/2$-$\mathsf{DMO}_{p'}$ for $p > 2$, and $\mathsf{DMO}_{2+\kappa}$ for $1 < p \le 2$. By a similar computation, if $F=|G|^{p-2}G$ and $G$ is $\theta_p$-Dini continuous, then $F$ is $p'/2$-$\mathsf{DMO}_{p'}$ for $p > 2$, and $\mathsf{DMO}_{2+\kappa}$ for $1 < p \le 2$. Consequently, Condition~\ref{datacondi} is weaker than the Dini condition in \cite{Ok17}.

Suppose $f\in L^{n,1}(B_r)$. Then one can find a unique solution $v\in W^{2,1}(B_r)\cap W^{1,1}_0(B_r)$ with $D^2v\in L^{n,1}(B_r)$ to the Poisson equation $\dv (Dv) = \Delta v=f$ in $B_r$. Set $F=Dv$. Then $F\in W^{1,1}(B_r)$ and $\dv F\in L^{n,1}(B_r)$. Hence, by Remark~\ref{rmk:Dini} (iv), $F$ is $\mathsf{DMO}_q$ for any $q\ge 1$. Therefore, Condition~\ref{datacondi} is weaker than the Lorentz condition for the source term in \cite{KM14} when $1< p\le 2$.
\end{remark}

Regarding the mean oscillation function $\omega_{\mathsf{data}}$ satisfying the Dini condition, we will use the following estimate, which can be found in \cite[Lemma 2.7]{DK17} and \cite[Lemma 3.6]{HLO25}.
\begin{lemma}\label{lem:Dini}
For $\delta\in(0,1)$, there exists $C_\omega=C_\omega>0$ depending on $n$, $p$, $\kappa$ and $\delta$ such that for every $r\in(0,1]$,
\begin{equation*}
\sum^\infty_{i=0} \omega_{\mathsf{data}}(\delta^i r) \leq C_\omega \int^r_0 \frac{\omega_{\mathsf{data}}(t)}{t}\,dt.
\end{equation*}
\end{lemma}

\subsection{Inequalities for \texorpdfstring{$p$}{p}-Laplacian type nonlinearties}
We recall the symmetric coefficient matrix $A$ in \eqref{elliptic}, and set
\begin{align*}
\mathbf{a}(x,\xi) := \ip{A(x)\xi,\xi}^{\frac{p-2}{2}} A(x)\xi, \qquad
\bar{\mathbf{a}}^{x_0,r}(\xi) := \ip{\bar{A}^{x_0,r}\xi,\xi}^{\frac{p-2}{2}} \bar{A}^{x_0,r} \xi .
\end{align*}
for every $\xi$, $\eta\in \bR^n$ and $x_0\in\Omega$.
From the ellipticity condition \eqref{elliptic} of $A$, elementary computations (see, e.g., \cite{KZ99,KM14}) show that for every $x,y\in\Omega$ and $\xi,\eta\in\bR^N$,
\begin{align} \label{eq0130fri1}
\Lambda^{-\frac{p}{2}}\abs{\xi}^p \leq \mathbf{a}(x,\xi)\cdot\xi \leq \Lambda^{\frac{p}{2}} \abs{\xi}^p, \quad \text{and hence }\ |\mathbf{a}(x,\xi)| \leq \Lambda^{\frac{p}{2}} \abs{\xi}^{p-1},
\end{align}
\begin{align} \label{eq0130fri3} 
\min\{1,p-1\} \Lambda^{-\frac{p}{2}}\abs{\xi}^{p-2}\abs{\eta}^2 \leq \ip{D_\xi\mathbf{a}(x,\xi)\eta,\eta} \leq \max\{1,p-1\}\Lambda^{\frac{p}{2}}\abs{\xi}^{p-2}\abs{\eta}^2,
\end{align}
\begin{align} \label{eq0202thr1}
\abs{\mathbf{a}(x,\xi) - \mathbf{a}(y,\xi)} &\leq C\abs{ A(x) -A(y) } \,\abs{\xi}^{p-1},
\end{align}
where $C=C(n,p,\Lambda)>0$. Note that the first inequality in \eqref{eq0130fri3} called the ellipticity condition for $\mathbf{a}(x,\xi)$ yields the following monotonicity condition:
\begin{align}\label{eq0127tue2}
( \mathbf{a}(x,\xi)-\mathbf{a}(x,\eta))\cdot(\xi-\eta)&\geq C(\abs{\xi}+\abs{\eta})^{p-2} \abs{\xi-\eta}^{2}.
\end{align}
Moreover, from a well-known variational inequality (see \cite[Lemma 3]{DE08}) and the symmetry of $A$, we also have that
\begin{align}\label{eq0127tue3}
\abs{ \mathbf{a}(x,\xi)-\mathbf{a}(x,\eta)} \sim_C (\abs{\xi}^2+\abs{\eta}^2)^{\frac{p-2}{2}}\abs{\xi-\eta}\sim_C \big| \abs{\xi}^{p-2}\xi-\abs{\eta}^{p-2}\eta \big|
\end{align}
where $C=C(n,p,\Lambda)>0$. Here, the notation $f \sim_C g $ means that there exists a constant $C>0$ such that $C^{-1} g \leq f \leq C g$. We remark that the same inequalities hold with $\mathbf{a}(x,\xi)$ replaced by $\bar{\mathbf{a}}^{x_0,r}(\xi)$.
  
Also, we consider the vector field $\xi \to \abs{\xi}^{\frac{p-2}{2}}\xi$. Then, we have that for every $\xi$, $\eta\in \bR^n$,
 \begin{align}\label{eq0130fri4}
 \frac{1}{C} (\abs{\xi}+\abs{\eta})^{\frac{p-2}{2}} \leq \frac{\big|\abs{\xi}^{\frac{p-2}{2}}\xi-\abs{\eta}^{\frac{p-2}{2}}\eta\big|}{\abs{\xi-\eta}}\leq C(\abs{\xi}+\abs{\eta})^{\frac{p-2}{2}},
 \end{align}
 where $C=C(n,p)\ge 1$ (see, e.g., \cite{KM14}).

\subsection{Preliminary regularity results}
We begin by presenting the integrability results for the gradient $Du$ of solutions to the variational inequality \eqref{main_vari ineq}. 
We first recall the reverse H\"{o}lder type inequality in \cite[Theorem~3.1]{Ok17}.
\begin{lemma}
Let $u\in \mathcal{A}_\psi$ is a solution to  the variational inequality \eqref{main_vari ineq},
and suppose that  $\abs{D\psi}$, $\abs{F} $ in $L^\infty_{\loc}(\Omega)$.
Then, there exists $\sigma_0=\sigma_0(n,p,\Lambda)>0$ such that for any ball $B_{2r}\Subset \Omega$, $\sigma\in (0,\sigma_0]$, and $t\in(0,1]$, we have that
\begin{align}\label{eq0330mon1}
\bigg(  \fint_{B_r} \abs{Du}^{p(1+\sigma)}\,dx \bigg)^{\frac{1}{1+\sigma}} \leq C_t  \Biggset{ \bigg( \fint_{B_{2r}} \abs{Du}^{pt} \, dx\bigg)^\frac{1}{t}+    \|F\|_{L^\infty(B_{2r}}^{\frac{p}{p-1}}  +\|D\psi\|_{L^\infty(B_{2r})}^p}.
\end{align}
where $C=C(n,p,\Lambda,t)$.
\end{lemma}

We next introduce Calder\'{o}n--Zygmund type estimates for the variational inequality \eqref{main_vari ineq}.
As mentioned above, Condition~\ref{datacondi} implies that $A$ is a \textsf{VMO} function. Therefore, in view of  \cite[Theorem~1.5]{BDM11} and \cite[Theorem~2.5]{BCW12}, combined with the result in the previous lemma, we can deduce the following result:

\begin{lemma}\label{lem:CZ_variational inequality}
Let $u\in \mathcal{A}_\psi$ is a solution to  the variational inequality \eqref{main_vari ineq} under the assumptions  \eqref{eq0130fri1}, and \eqref{eq0130fri3} with $p>1$. Suppose that the mean oscillation of the  $\mathsf{data}$ satisfies the Dini condition. 
Then for any $q>p$ and $t\in(0,1]$, there exist $\tilde R_q, \tilde C_q > 0$ depending on $n$, $p$, $\Lambda$, $q$, and $\tilde C_{q,t} > 0$ depending additionally on $t$, such that for any ball $B_{2r}\Subset \Omega$ with $r\le \tilde R_q$, we have
\begin{align}\label{eq0128wed2} 
\fint_{B_r}\abs{Du}^q\, dx &\leq  \tilde C_{q} \Biggset{     \bigg( \fint_{B_{3r/2}} \abs{Du}^{p} \, dx\bigg)^\frac{q}{p}+ \fint_{B_{3r/2}}  \big( \abs{F}^{\frac{q}{p-1}}  +\abs{D\psi}^q  \big)\, dx} \nonumber\\
&\leq  \tilde C_{q,t} \Biggset{     \bigg( \fint_{B_{2r}} \abs{Du}^{pt} \, dx\bigg)^\frac{q}{pt}+ \|F\|_{L^\infty(B_{2r}}^{\frac{q}{p-1}}  +\|D\psi\|_{L^\infty(B_{2r})}^q }.
\end{align}  
\end{lemma}
 
Let us consider the following $p$-Laplacian type equation involving a constant symmetric matrix $A_0$ satisfying \eqref{elliptic} with $A(x)$ replaced by $A_0$:
\begin{equation} \label{eq0427mon1}
-\dv ( \ip{A_0 D \tilde{u}, D\tilde{u} }^{\frac{p-2}{2}}A_0 D\tilde{u})=-\dv F \quad\hbox{in $B_r$}.
\end{equation}
The first regularity result is the global Calder\'{o}n--Zygmund estimate with nonzero boundary data (see \cite[Theorem~5.1]{KZ01}).
\begin{lemma}
Let $g\in W^{1,p}(B_r)$ and $\tilde u\in W^{1,p}(B_r)$ be the weak solution to \eqref{eq0427mon1} with $u-g\in W^{1,p}_0(B_r)$.  Suppose that $F\in L^q(B_r)$ and $g\in W^{1,q}(B_r)$ for some $q>p$.
There exists a constant $C=C(n,p,\Lambda,q)>0$ so that
\begin{align}\label{eq0410fri1}
\int_{B_r} \abs{D\tilde{u}}^q\, dx &\leq  C\int_{B_r}  \big( \abs{F}^{\frac{q}{p-1}}  +\abs{Dg}^q  \big)\, dx   .
\end{align}
 
\end{lemma}

We next present local boundedness and excess decay estimates for $D\tilde{u}$ associated with the $C^{1,\alpha}$-regularity of \eqref{eq0427mon1} in the $L^1$ sense. While the $C^{1,\alpha}$-regularity for the $p$-Laplacian type equation is classical (see, e.g., \cite{Evans82, DiBenedetto83}), the following result provides a recent improvement due to Antonini \cite[Theorem~4.1]{Antonini26}.

\begin{lemma}\label{lem:C1alpha}
 Let $\tilde{u}\in W^{1,p}(B_r)$ be a weak solution to \eqref{eq0427mon1}.  There exist $\alpha\in (0,1)$ and $C_0>0$ depending on $n$, $p$, and $\Lambda$ such that, for any $B_{\nu}(y)\subset B_\rho(y)\subset B_r$, 
\begin{align} \label{eq0331tue2}
\begin{split}
\norm{D\tilde{u}}_{L^\infty(B_{\rho/2}(y))} &\leq C_0\fint_{B_\rho(y)}\abs{D\tilde{u}}\,dx,\\
\fint_{B_\nu(y)} \abs{ D\tilde{u}- \overline{D\tilde{u}}^{y,\nu}}\,dx\leq &\osc_{B_\nu(y)} Dv  \leq C_0\bigg(\frac{\nu}{\rho}\bigg)^\alpha \fint_{B_\rho(y)} \abs{ D\tilde{u}- \overline{D\tilde{u}}^{y,\rho}}\,dx.
\end{split}
\end{align}
\end{lemma}

\section{Comparison Estimates}  \label{sec3}

Let $u\in \mathcal{A}_\psi$ be a solution to the variational inequality \eqref{main_vari ineq} and $B_{2r}(x_0)\Subset \Omega$, and the mean oscillation of the $\mathsf{data}$ satisfy the Dini condition with constant $\kappa>0$.
We take 
\begin{equation}\label{eq:R0}
R_0:=\min\{1,\tilde R_q\} \quad \text{with }\ q=\max\left\{p\frac{2+\kappa}{\kappa},\frac{(2-p)(2+\kappa)}{\kappa}\right\},
\end{equation}
where $\tilde R_q $ is given in Lemma~\ref{lem:CZ_variational inequality}.
We then consider the following two comparison problems:
\begin{equation} \label{eq0201sun1}
\begin{cases}
-\dv ( \ip{\bar{A}^{x_0,r}Dw, Dw }^{\frac{p-2}{2}} \bar{A}^{x_0,r} Dw)
=- \dv ( \ip{\bar{A}^{x_0,r} D\psi, D\psi }^{\frac{p-2}{2}} \bar{A}^{x_0,r} D\psi) 
& \text{in } B_r(x_0), \\
 w= u & \text{on } \partial B_r(x_0),
\end{cases}
\end{equation}
\begin{equation}\label{eq0201sun2}
\begin{cases}
-\dv ( \ip{\bar{A}^{x_0,r}Dv, Dv }^{\frac{p-2}{2}} \bar{A}^{x_0,r} Dv)
 =0
& \text{in } B_r(x_0), \\
 v= w & \text{on } \partial B_r(x_0),
\end{cases}
\end{equation}
where $B_{2r}(x_0)\Subset\Omega$ and $r\le R_0$.
Let $w$, $v\in W^{1,p}(B_r(x_0))$ are the weak solutions to the problems. We first derive energy estimates.
\begin{lemma}  
Under the above setting, we have
\begin{align} \label{eq0410fri3}
\begin{split}
\fint_{B_r(x_0)}\abs{Dw}^p \,dx &\leq C\bigg( \fint_{B_r(x_0)}  \abs{D\psi}^p+ \abs{Du}^p\, dx \bigg)\\
\fint_{B_r(x_0)}\abs{Dv}^p \,dx &\leq C\fint_{B_r(x_0)}\abs{Dw}^p \,dx \leq C\bigg( \fint_{B_r(x_0)}  \abs{D\psi}^p+ \abs{Du}^p\, dx \bigg)
\end{split}
\end{align}
where $C=C(n,p,\Lambda)>0$. Moreover, 
\begin{equation}\label{eq0410fri2}
 \fint_{B_r(x_0)}\abs{Dv}\,dx \leq C_1 \fint_{B_{2r}(x_0)} \abs{Du} \,dx+  \norm{F}_{L^\infty(B_{2r}(x_0))}^{\frac{1}{p-1}}+\norm{D\psi}_{L^\infty(B_{2r}(x_0))}    
\end{equation}
 where $C_1=C_1(n,p,\Lambda,\kappa)\geq 1$.

\end{lemma}
\begin{proof}
Taking $w-u\in W^{1,p}_0(B_r(x_0))$ as a test function in \eqref{eq0201sun1}, we obtain
\begin{align} \label{eq0201sun3}
\begin{split}
&\fint_{B_r(x_0)} \ip{ \bar{A}^{x_0,r}  Dw,Dw}^{\frac{p-2}{2}} \bar{A}^{x_0,r} Dw\cdot D(w-u) \,dx \\
&\qquad\qquad = \fint_{B_r(x_0)}  \ip{ \bar{A}^{x_0,r} D\psi,D\psi}^{\frac{p-2}{2}} \bar{A}^{x_0,r} D\psi\cdot D(w-u)  \,dx
\end{split}
\end{align}
By \eqref{eq0130fri1}, \eqref{eq0201sun3} and Yonung's inequality, we have that
\begin{align*}
\fint_{B_r(x_0)} \abs{Dw}^p \, dx &\leq C\fint_{B_r(x_0)} \ip{\bar{A}^{x_0,r} Dw,Dw}^{\frac{p-2}{2}} \bar{A}^{x_0,r}Dw\cdot Dw \,dx\\
&\leq C\fint_{B_r(x_0)} \ip{\bar{A}^{x_0,r} Dw,Dw}^{\frac{p-2}{2}}\bar{A}^{x_0,r}Dw\cdot Du\, dx \\
&\qquad +C\fint_{B_r(x_0)}  \ip{\bar{A}^{x_0,r}  D\psi,D\psi}^{\frac{p-2}{2}} \bar{A}^{x_0,r}D\psi\cdot D(w-u) \,dx \\
&\leq \frac{1}{2}  \fint_{B_r(x_0)}  \abs{Dw}^p\, dx+C  \fint_{B_r(x_0)}  \abs{D\psi}^p+ \abs{Du}^p \,dx.
\end{align*}
This proves the first of \eqref{eq0410fri3}. Similarly, by taking $v-w\in W^{1,p}_0(B_r(x_0))$ as a test in \eqref{eq0201sun2},
we have that
\begin{align*}
\fint_{B_r(x_0)} \ip{ \bar{A}^{x_0,r}  Dv,Dv}^{\frac{p-2}{2}} \bar{A}^{x_0,r} Dv\cdot D(v-w) \,dx=0.
\end{align*}
By \eqref{eq0130fri1} and Young's inequality, we also prove the second of \eqref{eq0410fri3}.

Finally, by H\"{o}lder's inequality, \eqref{eq0410fri3}, \eqref{eq0330mon1} with $t=1/p$ and qusai-triangle inequality, we also proves the \eqref{eq0410fri2}.
\end{proof}

Now, we shall derive comparison estimates between $Du$ and $Dv$.
\begin{lemma}  \label{lemma0205thr1}
Under the setting in the beginning of the section, we have 
\begin{align*}
&\fint_{B_r(x_0)} \big(  \abs{Du}+\abs{Dv} \big)^{p-2}\abs{Du-Dv}^2 \,dx\\
&\leq C\omega_{\mathsf{data}}(r,x_0)^2\biggset{    \bigg(\fint_{B_{2r}(x_0)} \abs{Du} \,dx\bigg)^p +    \norm{F}_{L^\infty(B_{2r}(x_0))}
^{p'}+\norm{D\psi}_{L^\infty(B_{2r}(x_0))}^p +1      }
\end{align*}
where $C=C(n,\Lambda,p,\kappa)>0$. 
\end{lemma}
\begin{proof}
By noting that $w=u\geq \psi$ on $\partial B_r(x_0)$ and by comparison principle in \eqref{eq0201sun1},  we observe that $w\geq \psi $ a.e. in $B_r(x_0)$, see e.g. \cite[Lemma 3.5]{BCW12}. Thus, setting $w:=u$ in $\Omega\backslash B_r(x_0)$, we see that $w\in \mathcal{A}_\psi$.
 
Note that we can rewrite the following using the right-hand side in \eqref{eq0201sun1}
\begin{equation}   \label{eq0205thr2}
\begin{split}
&-\dv ( \ip{\bar{A}^{x_0,r}Dw, Dw }^{\frac{p-2}{2}} \bar{A}^{x_0,r} Dw)\\
&=- \dv \bigg( \ip{\bar{A}^{x_0,r} D\psi, D\psi }^{\frac{p-2}{2}} \bar{A}^{x_0,r} D\psi - \overline{\ip{\bar{A}^{x_0,r}D\psi, D\psi}^{\frac{p-2}{2}} \bar{A}^{x_0,r} D\psi}^{x_0,r}+\bar{F}^{x_0,r}\bigg) 
\end{split}
\end{equation}
Then by taking $w-u\in W^{1,p}_0(B_r(x_0))$ as a test function in \eqref{eq0205thr2} and taking $\varphi=w$ in \eqref{main_vari ineq}, we obtain that
\begin{align*}
& \fint_{B_r(x_0)} \big( \ip{ \bar{A}^{x_0,r}  Dw,Dw}^{\frac{p-2}{2}} \bar{A}^{x_0,r} Dw -\ip{ \bar{A}^{x_0,r}  Du,Du}^{\frac{p-2}{2}} \bar{A}^{x_0,r} Du \big) \cdot (Dw-Du) \,dx \\
& \leq \fint_{B_r(x_0)} \big( \ip{ A(x) Du,Du}^{\frac{p-2}{2}} A(x) Du -\ip{ \bar{A}^{x_0,r}  Du,Du}^{\frac{p-2}{2}} \bar{A}^{x_0,r} Du \big) \cdot (Dw-Du) \,dx 
\\ &\quad  + \fint_{B_r(x_0)}  \bigg( \ip{ \bar{A}^{x_0,r} D\psi,D\psi}^{\frac{p-2}{2}} \bar{A}^{x_0,r} D\psi   -\overline{\ip{\bar{A}^{x_0,r}D\psi, D\psi}^{\frac{p-2}{2}} \bar{A}^{x_0,r} D\psi}^{x_0,r} \bigg) \cdot  (Dw-Du)\, dx
\\ &\quad  + \fint_{B_r(x_0)} (\bar{F}^{x_0,r}-F(x)) \cdot  (Dw-Du) \,dx
\\ & :=I_1+I_2+I_3.
\end{align*}
Applying \eqref{eq0127tue2} and \eqref{eq0130fri4} to the left-hand side of the above inequality, we have
\begin{equation} \label{eq0304wed1}
\frac{1}{\tilde C_0}\fint_{B_r(x_0)} \big(  \abs{Du}+\abs{Dw} \big)^{p-2}\abs{Du-Dw}^2 \,dx\leq I_1+I_2+I_3
\end{equation}
for some $\tilde C_0\ge 1$.

We first estimate $I_1$. By using \eqref{eq0202thr1} and  Young's inequality,  we have
\begin{align*}
I_1&\leq C  \fint_{B_r(x_0)} \abs{ A(x) -\bar{A}^{x_0,r} } \abs{Du}^{p-1} \abs{Dw-Du} \,dx \\
 &\leq C  \fint_{B_r(x_0)} \abs{ A(x) -\bar{A}^{x_0,r} }  \big(  \abs{Du} +\abs{Dw} \big)^{\frac{p}{2} +\frac{p-2}{2}} \abs{Dw-Du}\, dx\\
 &\leq C  \fint_{B_r(x_0)} \abs{ A(x) -\bar{A}^{x_0,r} } ^2\big(  \abs{Du} +\abs{Dw} \big)^p \,dx \\
 &\qquad\qquad\qquad\qquad+ \frac{1}{6\tilde C_0} \fint_{B_r(x_0)} \big(  \abs{Du}+\abs{Dw} \big)^{p-2}\abs{Du-Dw}^2 \,dx
\end{align*}
Moreover, from the definition of $L^{2+\kappa}$-mean oscillation of $A$, \eqref{eq0410fri1}, and \eqref{eq0128wed2} with $q=p\frac{2+\kappa}{\kappa}$ and $t=1/p,$  we further have
\begin{align*}
& \fint_{B_r(x_0)} \abs{ A(x) -\bar{A}^{x_0,r} } ^2\big(  \abs{Du} +\abs{Dw} \big)^p \,dx\\ 
& \leq \bigg(  \fint_{B_r(x_0)} \abs{ A(x) -\bar{A}^{x_0,r} }^{2+\kappa} \, dx\bigg) ^\frac{2}{2+\kappa} \bigg(  \fint_{B_r(x_0)} \big(  \abs{Du} +\abs{Dw} \big)^{p\frac{2+\kappa}{  \kappa}}\,dx \bigg) ^\frac{ \kappa}{2+\kappa}\\
&\leq \omega^{(2+\kappa)}_A(r,x_0)^2  \bigg( \fint_{B_{r}(x_0)} \abs{Du}^{p\frac{2+\kappa}{ \kappa}} + \abs{D\psi}^{p\frac{2+\kappa}{\kappa}} \, dx \bigg)^{\frac{\kappa}{2+\kappa}}  \\
&\leq \omega^{(2+\kappa)}_A(r,x_0)^2 \biggset{  \bigg(\fint_{B_{2r}(x_0)}   \abs{Du}  \, dx\bigg)^p+  \norm{F}_{L^\infty(B_{2r}(x_0))}
^{p'}+\norm{D\psi}_{L^\infty(B_{2r}(x_0))}^p}.
\end{align*}

We next estimate $I_2$. Suppose $p\geq 2$.
Then, by using \eqref{eq0127tue3}, Young's inequality, and definition of $L^{p'}$-mean oscillation for obstacle $\psi$, we have
\begin{align*}
I_2
&\leq C \fint_{B_r(x_0)}  \fint_{B_r(x_0)}\big| \ip{\bar{A}^{x_0,r} D\psi(x), D\psi (x)}^{\frac{p-2}{2}} \bar{A}^{x_0,r} D\psi(x)\\
&\qquad \qquad\qquad\qquad\quad  - \ip{\bar{A}^{x_0,r} D\psi(y), D\psi(y) }^{\frac{p-2}{2}} \bar{A}^{x_0,r} D\psi(y) \big| \,\abs{Dw(x)-Du(x)}\,dx\,dy\\
&\leq C \fint_{B_r(x_0)}  \fint_{B_r(x_0)}   \big| \abs{D\psi(x)}^{p-2}D\psi(x)-\abs{D\psi(y)}^{p-2}D\psi(y) \big|\abs{Dw(x)-Du(x)}\,dx\,dy\\
&\leq C\fint_{B_r(x_0)}  \fint_{B_r(x_0)}   \big| \,\abs{D\psi(x)}^{p-2}D\psi(x)-\abs{D\psi(y)}^{p-2}D\psi(y) \big|^{p'}\,dx\,dy\\
&\qquad  \qquad\qquad\qquad \qquad\qquad\qquad\qquad\quad +\frac{1}{6\tilde C_0} \fint_{B_r(x_0)}   \abs{Du-Dw}^p \,dx\\
&\leq  C \bigg(  \fint_{B_r(x_0)}  \big| \abs{D\psi(y)}^{p-2}D\psi(y)-\overline{\abs{D\psi}^{p-2}D\psi}^{x_0,r} \big|^{p'}\,dy\bigg)^{\frac{1}{p'}\frac{p'}{2}2}\\
&\qquad \qquad \qquad\qquad \qquad\qquad\qquad\quad +\frac{1}{6\tilde C_0} \fint_{B_r(x_0)} \big(  \abs{Du}+\abs{Dw} \big)^{p-2}\abs{Du-Dw}^2\, dx\\
&\leq  C \omega_{D\psi}^{(p')}(r,x_0)^{\frac{p'}{2}2}+ \frac{1}{6 \tilde C_0} \fint_{B_r(x_0)} \big(  \abs{Du}+\abs{Dw} \big)^{p-2}\abs{Du-Dw}^2 \,dx.
\end{align*}
On the other hand, if $1<p<2$,  by a similar argument as above, together with \eqref{eq0127tue3}, and using \eqref{eq0130fri4}, and Young's inequality, we obtain
\begin{align*}
I_2&\leq   C\fint_{B_r(x_0)}  \fint_{B_r(x_0)}   \big| \abs{D\psi(x)}^{p-2}D\psi(x)-\abs{D\psi(y)}^{p-2}D\psi(y) \big|^{2} \big( \abs{Du(x)}+\abs{Dw(x)}\big)^{2-p}\,dx\,dy \\
&\qquad\qquad\qquad\qquad+ \frac{1}{6\tilde C_0} \fint_{B_r(x_0)} \big(  \abs{Du}+\abs{Dw} \big)^{p-2}\abs{Du-Dw}^2 \,dx.
\end{align*}
By H\"{o}lder's inequality, using the definition of $L^{2+\kappa}$-mean oscillation of $|D\psi|^{p-2}D\psi$, \eqref{eq0410fri1}, \eqref{eq0128wed2} with $q= \frac{(2-p)(2+\kappa)}{\kappa}$ and $t=1/p$, and \eqref{eq0130fri1}, we have
\begin{align*}
&\fint_{B_r(x_0)}  \fint_{B_r(x_0)}   \big| \abs{D\psi(x)}^{p-2}D\psi(x)-\abs{D\psi(y)}^{p-2}D\psi(y) \big|^{2} \big( \abs{Du(x)}+\abs{Dw(x)}\big)^{2-p}\,dx \,dy  \\
& \leq C \fint_{B_r(x_0)}   \bigg( \fint_{B_r(x_0)} \big| \abs{D\psi(x)}^{p-2}D\psi(x)-\abs{D\psi(y)}^{p-2}D\psi(y) \big|^{2+\kappa} dx \bigg)^{\frac{2}{2+\kappa}} \, dy \\
&\qquad\qquad\qquad\qquad \qquad\qquad\qquad\qquad \times\bigg( \fint_{B_r(x_0)} \big( \abs{Du}+\abs{Dw}\big)^{\frac{(2-p)(2+\kappa)}{\kappa}} \,dx \bigg)^{\frac{ \kappa}{2+\kappa}} \\
&\leq C\bigg[   \bigg( \fint_{B_r(x_0)} \big| \abs{D\psi(x)}^{p-2}D\psi(x)-\overline{\abs{D\psi}^{p-2}D\psi}^{x_0,r} \big|^{2+\kappa}  \, dx  \bigg)^{\frac{2}{2+\kappa}}\\
&\qquad\qquad\qquad\qquad\qquad\qquad+  \fint_{B_r(x_0)}\big| \abs{D\psi(y)}^{p-2}D\psi(y)-\overline{\abs{D\psi}^{p-2}D\psi}^{x_0,r} \big|^2\, dy \bigg]\\
&\qquad\qquad\qquad\qquad \qquad\qquad\qquad\qquad \times\bigg( \fint_{B_r(x_0)} \big( \abs{Du}+\abs{D\psi}\big)^{\frac{(2-p)(2+\kappa)}{\kappa}} \,dx \bigg)^{\frac{ \kappa}{2+\kappa}} \\
&\leq C\omega_{|D\psi|^{p-2}D\psi}^{(2+\kappa)}(r,x_0)^2 \biggset{    \bigg(\fint_{B_{2r}(x_0)} \abs{Du}\, dx\bigg)^p +  \norm{F}_{L^\infty(B_{2r}(x_0))}
^{p'}+\norm{D\psi}_{L^\infty(B_{2r}(x_0))}^p   +1  },
\end{align*}
where we use the inequality $0<2-p<p$ and assume that $\frac{(2-p)(2+\kappa)}{\kappa}>p$ without loss generality.

We finally estimate $I_3$. In the same way to estimate $I_2$ with $\ip{\bar{A}^{x_0,r} D\psi, D\psi}^{\frac{p-2}{2}} \bar{A}^{x_0,r} D\psi$ and
$\overline{\ip{\bar{A}^{x_0,r}D\psi, D\psi}^{\frac{p-2}{2}} \bar{A}^{x_0,r} D\psi}^{x_0,r}$ replaced by $F$ and $\bar{F}^{x_0,r}$, respectively, we can obtain that, when $p\ge 2$, 
\begin{align*}
I_3
\leq C \omega_F^{(p')}  (r,x_0)  ^{\frac{p'}{2}2 }+\frac{1}{6\tilde C_0} \fint_{B_r(x_0)} \big(  \abs{Du}+\abs{Dw} \big)^{p-2}\abs{Du-Dw}^2  \, dx, 
\end{align*}
and, when $1<p<2$,   
\begin{align*}
I_3&\leq  C\omega_F^{(2+\kappa)}(r,x_0)^2 \biggset{    \bigg(\fint_{B_{2r}(x_0)} \abs{Du}\, dx\bigg)^p+   \norm{F}_{L^\infty(B_{2r}(x_0))}
^{p'}+\norm{D\psi}_{L^\infty(B_{2r}(x_0))}^p +1   }  \\
&\qquad +\frac{1}{6\tilde C_0} \fint_{B_r(x_0)} \big(  \abs{Du}+\abs{Dw} \big)^{p-2}\abs{Du-Dw}^2 \, dx.
\end{align*}

Consequently, by combining the above estimates, we obtain
\begin{align} \label{eq0304wed2}
\begin{split}
&\fint_{B_r(x_0)} \big(  \abs{Du}+\abs{Dw} \big)^{p-2}\abs{Du-Dw}^2 \,dx\\
&\leq C \omega_{\mathsf{data}}(r,x_0)^2\biggset{    \bigg(\fint_{B_{2r}(x_0)} \abs{Du} \, dx\bigg)^p +  \norm{F}_{L^\infty(B_{2r}(x_0))}
^{p'}+\norm{D\psi}_{L^\infty(B_{2r}(x_0))}^p+1    }
\end{split}
\end{align}
where $C=C(n,p,\Lambda)>0$.

Now, we derive a comparison estimate between $Dw$ and $Dv$. By taking $w-v\in W^{1,p}_0(B_r(x_0))$ as a test function in \eqref{eq0201sun1} and in \eqref{eq0201sun2} and using \eqref{eq0127tue2}, we have
\begin{align*}
&\fint_{B_r(x_0)} \big(  \abs{Dw}+\abs{Dv} \big)^{p-2}\abs{Dw-Dv}^2 \,dx\\
& \le C \fint_{B_r(x_0)} \big( \ip{ \bar{A}^{x_0,r}  Dw,Dw}^{\frac{p-2}{2}} \bar{A}^{x_0,r} Dw -\ip{ \bar{A}^{x_0,r}  Dv,Dv}^{\frac{p-2}{2}} \bar{A}^{x_0,r} Dv \big) \cdot (Dw-Dv) \,dx \\
&=C\fint_{B_r(x_0)}  \bigg( \ip{ \bar{A}^{x_0,r} D\psi,D\psi}^{\frac{p-2}{2}} \bar{A}^{x_0,r} D\psi   -\overline{\ip{\bar{A}^{x_0,r}D\psi, D\psi}^{\frac{p-2}{2}} \bar{A}^{x_0,r} D\psi}^{x_0,r}\bigg) \cdot  (Dw-Dv) \,dx .
\end{align*}
Then, applying the same estimation of $I_2$ in \eqref{eq0304wed1} to the right hand side of the above estimate, we can estimate that
\begin{equation}\label{eq0304wed3}
\begin{split}
&\fint_{B_r(x_0)} \big(  \abs{Dw}+\abs{Dv} \big)^{p-2}\abs{Dw-Dv}^2 \,dx \\
&\leq  C \bar{\omega}_{|D\psi|^{p-2}D\psi}(r,x_0)^{2} \biggset{  \bigg(\fint_{B_r(x_0)} \abs{Du} \, dx \bigg)^p   +   \norm{F}_{L^\infty(B_{2r}(x_0))}
^{p'}+\norm{D\psi}_{L^\infty(B_{2r}(x_0))}^p+1     }
\end{split}
\end{equation}
for some $C=C(n,p, \Lambda)>0$, where $\bar{\omega}_{\cdot}(r,x_0):= \omega_{\cdot}^{(p')}(r,x_0)^{\frac{p'}{2}}$ if $p\ge 2$ and $\bar{\omega}_{\cdot}(r,x_0):= \omega_{\cdot}^{(2+\kappa)}(r,x_0)$ if $1<p\le 2$.

Therefore, since  $ (  \abs{Du}+\abs{Dv} )^{p-2}\abs{Du-Dv}^2  \le C \{ ( \abs{Du}+\abs{Dw} )^{p-2}\abs{Du-Dw}^2  + (  \abs{Dw}+\abs{Dv} )^{p-2}\abs{Dw-Dv}^2 \}$  by \eqref{eq0130fri4}, we obtain the desired estimate  from \eqref{eq0304wed2} and \eqref{eq0304wed3}.
\end{proof}

From the previous lemma, we immediately obtain the following $L^1$ comparison estimate between $Du$ and $Dv$.
\begin{corollary} \label{coro0304wed1}
Under the setting in the beginning of the section,  we have
\begin{equation*}
\begin{aligned}
&\fint_{B_r(x_0)} \abs{ Du-Dv} \, dx\\
&\leq C_2 \omega_{\mathsf{data}}(r,x_0)^{\frac{\min\set{2,p}}{p}}\biggset{    \fint_{B_{2r}(x_0)} \abs{Du} \, dx +   \norm{F}_{L^\infty(B_{2r}(x_0))}
^{\frac{1}{p-1}}+\norm{D\psi}_{L^\infty(B_{2r}(x_0))} +1     }
\end{aligned}
\end{equation*}
 
where $C_2=C_2(n,p,\Lambda)\geq 1$.
\end{corollary}
\begin{proof}
The case $p\geq2$ follows trivially from the previous lemma using H\"{o}lder's inequality. On the other hand, if $1<p<2$, by using H\"{o}lder's inequality and \eqref{eq0410fri3}, we obtain
\begin{align*}
&\fint_{B_r(x_0)} \abs{ Du-Dv}^p\, dx =\fint_{B_r(x_0)} ( \abs{Du}+\abs{Dv})^{\frac{p(2-p)}{2}+\frac{p(p-2)}{2}}\abs{ Du-Dv}^p\, dx\\
&\leq C    \bigg(  \fint_{B_r(x_0)}  ( \abs{Du}+\abs{Dv})^p \, dx \bigg)^{\frac{2-p}{2}} \bigg( \fint_{B_r(x_0)} ( \abs{Du}+\abs{Dv})^{p-2}\abs{ Du-Dv}^2 \, dx \bigg)^{\frac{p}{2}}\\
&\leq C   \omega_{\mathsf{data}}(r,x_0)^2 \bigg(    \fint_{B_r(x_0)}  \abs{D\psi}^p+ \abs{Du}^p\,  dx  \bigg)^{\frac{2-p}{2}}\\
&\qquad\qquad\times\biggset{    \fint_{B_{2r}(x_0)} \abs{Du} \, dx +   \norm{F}_{L^\infty(B_{2r}(x_0))}
^{\frac{1}{p-1}}+\norm{D\psi}_{L^\infty(B_{2r}(x_0))} +1    }^{\frac{p}{2}}
\end{align*}
Combining this with H\"{o}lder's inequality and \eqref{eq0330mon1} with $t=1/p$, we obtain the desired estimate for the case $1<p<2$.
\end{proof}

The next lemma provides a sharper comparison estimate for $Du$ and $Dv$ under a lower bound assumption on $Dv$.

\begin{lemma} \label{lemma0327fri1}
Under the setting in the beginning of the section, there exists $C_3=C_3(n,p,\Lambda)>0$ such that if 
\[
\frac{M}{N} \leq \inf_{B_\rho(x_0)}\abs{Dv}  
\]
and
\[
  \fint_{B_{2r}(x_0)} \abs{Du}\, dx  +   \norm{F}_{L^\infty(B_{2r}(x_0))}
^{\frac{1}{p-1}}+\norm{D\psi}_{L^\infty(B_{2r}(x_0))} +1  \leq M
\]
for some $M$, $N\geq 1$ and $B_\rho(x_0)\subset  B_r(x_0)$, then
\[
 \fint_{B_\rho(x_0)}\abs{Du-Dv}\, dx\leq C_3 N^{p-1}\omega_{\mathsf{data}}(r,x_0)   \bigg(\frac{r}{\rho}\bigg)^n M
\]
\end{lemma}
\begin{proof}
Note that $\abs{Dv}+\abs{Du}\geq M/N$ in $B_\rho(x_0)$. By Young's inequality, \eqref{eq0410fri2}, Lemma \ref{lemma0205thr1}, and the second assumption, we obtain that
\begin{align*}
\fint_{B_\rho(x_0)}\abs{Du-Dv}  \, dx&= 
\bigg(\frac{r}{\rho}\bigg)^n   \fint_{B_t(x_0)}\big(\frac{N}{M}\big)^{\frac{p-1}{2}}(|Du|+|Dv|)^{\frac{p-1}{2}}\abs{Du-Dv}  \, dx\\
& \leq  \omega_{\mathsf{data}}(r,x_0)\bigg(\frac{r}{\rho}\bigg)^n  \fint_{B_r(x_0)} \big(\abs{Du}+\abs{Dv}\big)\,  dx\\
&\quad +\frac{N^{p-1}}{\omega_{\mathsf{data} } (r,x_0) M^{p-1}} \bigg(\frac{r}{\rho}\bigg)^n \fint_{B_r(x_0)}\big(\abs{Du}+\abs{Dv}\big)^{p-2}\abs{Du-Dv}^{2}\, dx\\
&\leq  \omega_{\mathsf{data}}(r,x_0) \bigg(\frac{r}{\rho}\bigg)^n  M+
+C \omega_{\mathsf{data}}(r,x_0) N^{p-1}  \bigg(\frac{r}{\rho}\bigg)^n M\\
&\leq C \omega_{\mathsf{data}}(r,x_0) N^{p-1}  \bigg(\frac{r}{\rho}\bigg)^n M. \qedhere
\end{align*}
 
\end{proof} 

\section{Regularity of the Gradient}  \label{sec4}

In this section, we prove the gradient continuity of the solution $u\in \mathcal{A}_\psi$ to the variational inequality \eqref{main_vari ineq}, that is Theorem \ref{main_theorem}. Hence, we assume that the mean oscillation of the $\mathsf{data}$ satisfies the Dini condition with constant $\kappa>0$.

We recall the constants $C_\omega$, $\alpha$,  $C_0$, $R_0$,  $C_1$, $C_2$, and $C_3$ from Lemma~\ref{lem:Dini}, Lemma~\ref{lem:C1alpha}, \eqref{eq:R0}, \eqref{eq0410fri2}, Corollary~\ref{coro0304wed1}, and Lemma~\ref{lemma0327fri1}. 
For $\epsilon\in(0,\epsilon_0]$,  where $\epsilon_0:=2^{-n-2}$, we choose 
$\delta=\delta(\epsilon) \in (0,\frac{1}{4})$ such that
\begin{align}\label{eq:delta1}
64C_0C_1\delta^\alpha\leq \epsilon
\end{align} 
and $R_1=R_1(\epsilon)>0$ such that for every $r\in(0,R_1]$,
\begin{equation}\label{eq:R11}
(2\delta)^{-n} C_2 \omega_{\mathsf{data}}(r)^{\frac{\min\set{2,p}}{p}}\le \frac{\epsilon}{2}, \quad 4\delta^{-n}C_2 \omega_{\mathsf{data}}(r)^{\frac{\min\set{2,p}}{p}}\le  \frac{\delta^n \epsilon_0 \epsilon}{16},
\end{equation}
and
\begin{equation}\label{eq:R12}
\sum^\infty_{i=0} \omega_{\mathsf{data}}(\delta^i r) \le C_\omega \int^{r}_0\frac{\omega_\mathsf{data}(t)}{t}\,dt \leq  \frac{\delta^n  \epsilon_0 \epsilon }{2^{2p+3}C_3\epsilon^{-(p-1)}\delta^{-2n} } .
\end{equation}

Fix $B_{2r}(x_0)\Subset\Omega$ with $r\le \min\{R_0,R_1\}$.  Then for each $j\in \mathbb{N}\cup \set{0}$, we set
\[
r_j:=\delta^j r,
\quad B_j:=B_{r_j}(x_0), 
\quad\hbox{and}\quad E_j=E_j(x_0,r,\delta):=\fint_{B_j}\abs{Du-\overline{Du}^{x_0,r_j}} \,dx,
\]
and   let $w_j$, $v_j\in W^{1,p}(B_j)$ be the weak solutions to the following problems:
\begin{equation*} 
\begin{cases}
-\dv ( \ip{\bar{A}^{x_0,r}Dw_j, Dw_j }^{\frac{p-2}{2}} \bar{A}^{x_0,r} Dw_j)
=- \dv ( \ip{\bar{A}^{x_0,r} D\psi, D\psi }^{\frac{p-2}{2}} \bar{A}^{x_0,r} D\psi) 
& \text{in } B_j, \\
 w_j= u & \text{on } \partial B_j,
\end{cases}
\end{equation*}

\begin{equation*}
\begin{cases}
-\dv ( \ip{\bar{A}^{x_0,r}Dv_j, Dv _j}^{\frac{p-2}{2}} \bar{A}^{x_0,r} Dv_j)
 =0
& \text{in } B_j, \\
 v_j= w_j & \text{on } \partial B_j.
\end{cases}
\end{equation*}

\begin{lemma}\label{lemma0331tue1}
Under the above setting, for any $j\in \mathbb{N}\cup\{0\}$, if
\[
\fint_{2B_j}\abs{Du} \, dx +  \norm{F}_{L^\infty(2B_j)}
^{\frac{1}{p-1}}+\norm{D\psi}_{L^\infty(2B_j)} +1\leq M
\]
and
\[
\fint_{2B_{j+1}}\abs{Du}\, dx\geq \epsilon M
\]
form some $M\geq1$, then
\[
E_{j+2}\leq \frac{\epsilon}{32} E_{j+1} +2^{2p}C_3\delta^{-2n} \epsilon^{-(p-1)}\omega_{\mathsf{data}}(\delta^j r)   M.
\]
 
\end{lemma}
\begin{proof}
Integrate the both sides of the following inequality over $B_{j+2}$
\[
\abs{Du-\overline{Du}^{x_0,r_{j+2} }}\leq \abs{Du-Dv_j}+\abs{ Dv_j- \overline{Dv_j}^{x_0,r_{j+2}}}+\abs{  \overline{Dv_j}^{x_0,r_{j+2}}-\overline{Du}^{x_0,r_{j+2} }},
\]
we get 
\[
E_{j+2}\leq \fint_{B_{j+2} } \abs{ Dv_j- \overline{Dv_j}^{x_0,r_{j+2}}} \,dx +2 \fint_{B_{j+2} }   \abs{Du-Dv_j}\,dx.
\]
Then by \eqref{eq0331tue2} and \eqref{eq:delta1}, we obtain
\begin{equation} \label{eq0331tue3}
\begin{aligned}
E_{j+2}&\leq \fint_{B_{j+2} } \abs{ Dv_j- \overline{Dv_j}^{x_0,r_{j+2}}} \,dx +2 \fint_{B_{j+2} }   \abs{Du-Dv_j}\,dx\\
&\leq C_0 \delta^\alpha \fint_{B_{j+1} } \abs{ Dv_j- \overline{Dv_j}^{x_0,r_{j+1}}} \,dx +2\delta^{-n} \fint_{B_{j+1} }   \abs{Du-Dv_j}\, dx\\
&\leq 2C_0 \delta^\alpha \fint_{B_{j+1} } \abs{ Dv_j- \overline{Du}^{x_0,r_{j+1}}} \,dx +2\delta^{-n} \fint_{B_{j+1} }   \abs{Du-Dv_j}\, dx\\
&\leq \frac{\epsilon}{32} E_{j+1}+4\delta^{-n}  \fint_{B_{j+1} }   \abs{Du-Dv_j}\,dx.
\end{aligned}
\end{equation}
To estimate the second integral on the right-hand side above, we observe from Corollary~\ref{coro0304wed1} that
\[
\epsilon  M\leq  \fint_{2B_{j+1}}\abs{Du-Dv_j}\, dx +\fint_{2B_{j+1}}\abs{Dv_j}\, dx\leq (2\delta)^{-n} C_2 \omega_{\mathsf{data}}(\delta^j r)^{\frac{\min\set{2,p}}{p}}M +\fint_{2B_{j+1}}\abs{Dv_j}\,dx,
\]
which together with \eqref{eq:R11} implies 
\[
\fint_{2B_{j+1}}\abs{Dv_j}\, dx \geq \frac{\epsilon M}{2}.
\]
So, by the above,  \eqref{eq0331tue2}, \eqref{eq0410fri2}, and \eqref{eq:delta1},  we have 
\begin{align*}
\inf_{2B_{j+1}}\abs{ Dv_j}
 \geq \fint_{2B_{j+1}}\abs{Dv_j}\, dx -\osc_{2B_{j+1}} \abs{Dv_j} 
&\geq \frac{\epsilon M}{2} -2C_0(2\delta)^{\alpha} \fint_{B_{j}} \abs{Dv_j}\, dx\\
& \geq \frac{\epsilon M}{2} -4 C_0 C_1 \delta^{\alpha}M  \ge \frac{\epsilon M}{4}.
\end{align*}
Therefore, applying Lemma \ref{lemma0327fri1} with the above, we obtain
\[
 \fint_{B_{j+1}}\abs{Du-Dv_j} \,dx\leq C_3 \bigg( \frac{4}{\epsilon}\bigg)^{p-1} \omega_{\mathsf{data}}(\delta^j r)    \bigg(\frac{1}{\delta}\bigg)^n M.
\]
Inserting the above into \eqref{eq0331tue3}, we have the desired estimate.
\end{proof}

We remark that we do not use the condition \eqref{eq:R12} for $r$ in the proof of the above lemma.

Next, we now turn to iteration estimates.

\begin{lemma} \label{lemma0413mon1}
Under the setting in the beginning of the section, suppose that for $k\geq j+1$ with $j\in \mathbb{N}\cup\{0\}$,
\begin{align}\label{eq0410fri4}
\frac{1}{\delta^n\epsilon}E_j +\fint_{2B_j}\abs{Du} \, dx +  \norm{F}_{L^\infty(2B_j)}
^{\frac{1}{p-1}}+\norm{D\psi}_{L^\infty(2B_j)} +1\leq 2\epsilon_0 M
\end{align}
and
\begin{align}\label{eq0410fri5}
\fint_{2B_{l} }\abs{Du}\, dx\geq \epsilon M  \quad\hbox{for every } \ j+1\leq l\leq k-1,
\end{align}
for some $M\geq 1$. Then,
\begin{align} \label{eq0412sun4}
E_k\leq \frac{\delta^n \epsilon_0 \epsilon}{2}M,\quad \fint_{B_{k+1}}\abs{Du -\overline{Du}^{x_0,r_j}} \, dx\leq 3\epsilon\epsilon_0  M,\quad\hbox{and}\quad\fint_{2B_k}\abs{Du}\,dx\leq M.
\end{align}
\end{lemma}
\begin{proof}
We prove the lemma by induction on $k$. First, the case $k=j+1$ holds. Indeed, applying the same estimation in \eqref{eq0331tue3} to $E_{j+1}$ instead of $E_{j+2}$, \eqref{eq0410fri4}, Corollary \ref{coro0304wed1}, \eqref{eq0410fri4}, and \eqref{eq:R11}, we obtain that
\begin{align}
E_{j+1}& \leq \frac{\epsilon}{32}E_j  +4\delta^{-n}  \fint_{B_{j } }   \abs{Du-Dv_j} \,dx \nonumber\\
&\leq \frac{\delta^n \epsilon_0 \epsilon }{16}M+4\delta^{-n}C_2 \omega_{\mathsf{data}}(\delta^j r)^{\frac{\min\set{2,p}}{p}}M \leq \frac{\delta^n \epsilon_0 \epsilon}{8}M. \label{eq0412sun1}
\end{align}
Using the above estimate for $E_{j+1}$ and \eqref{eq0410fri4}, we have
\begin{align*}
\fint_{B_{j+2}}\abs{Du -\overline{Du}^{x_0,r_j}} \, dx&\leq \fint_{B_{j+2}}\abs{Du -\overline{Du}^{x_0,r_{j+1}}} \,dx+\abs{\overline{Du}^{x_0,r_{j+1}}-\overline{Du}^{x_0,r_j}}\\
&\leq \delta^{-n}(E_{j+1}+E_j)\leq 3\epsilon\epsilon_0  M,
\end{align*}
and also obtain that
\begin{align*}
\fint_{2B_{j+1}}\abs{Du} \, dx  &\leq \fint_{2B_{j+1}}\abs{Du-\overline{Du}^{x_0,r_j}} \, dx+\abs{\overline{Du}^{x_0,r_j}}\\
&\leq (2\delta)^{-n}E_j+2^{n}\fint_{2B_j}\abs{Du}\, dx\leq M.
\end{align*}

Now, suppose that the claim \eqref{eq0412sun4} holds up to $k=k_0-1\ge j+1$. Using the Lemma \ref{lemma0331tue1} with $j$ replaced by $l\in\{j,\dots, k_0-2\}$, we obtain
\begin{equation*}
2E_{l+2}-E_{l+1}\leq 2^{2p+1}C_3\epsilon^{-(p-1)}\delta^{-2n} \omega_\mathsf{data}(\delta^l r)M
\end{equation*}
From this, \eqref{eq:R12}, and \eqref{eq0412sun1}, we can make summation with respect to $l$ to get
\begin{equation}\label{eq0412sun5}\begin{aligned} 
\sum^{k_0}_{l=j+1} E_l & =\sum^{k_0-2}_{l=j} (2E_{l+2}-E_{l+1}) +2E_{j+1}-E_{k_0}\\
&\leq 2^{2p+1}C_3\delta^{-2n} \epsilon^{-(p-1)} M  \sum^{k_0-2}_{l=j} \omega_{\mathsf{data}}(\delta^l r)+2 E_{j+1}\leq \frac{\delta^n \epsilon_0 \epsilon}{2}M.
\end{aligned}\end{equation}
So, this proves the first claim in \eqref{eq0412sun4} for $k=k_0$.
Next, by telescoping, we can write
\[
Du -\overline{Du}^{x_0,r_j}=\big( Du -\overline{Du}^{x_0,r_{k_0}}\big) +\big(  \overline{Du}^{x_0,r_{k_0}}-\overline{Du}^{x_0,r_{k_0-1}} \big)+\cdots + \big(  \overline{Du}^{x_0,r_{j+1}}-\overline{Du}^{x_0,r_{j}} \big).
\]
By taking the average over $B_{k_0+1}$ and using \eqref{eq0410fri4} and \eqref{eq0412sun5}, we get
\[
\fint_{B_{k_0+1}} \abs{Du -\overline{Du}^{x_0,r_j}} \, dx\leq \sum^{k_0}_{l=j}\fint_{B_{l+1}} \abs{Du -\overline{Du}^{x_0,r_l}} \, dx\leq \delta^{-n}\sum^{k_0}_{l=j} E_l\leq 3\epsilon\epsilon_0  M.
\]
This proves the second claim in \eqref{eq0412sun4}.
Finally, by similar estimate as in above, we observe that
\begin{align*}
\fint_{2B_{k_0}} \abs{Du -\overline{Du}^{x_0,r_j}}  \, dx
\leq (2\delta)^{-n}E_{k_0-1}+ \sum^{k_0-2}_{l=j}\fint_{B_{l+1}} \abs{Du -\overline{Du}^{x_0,r_l}}\, dx
& \leq(2\delta)^{-n}\sum^{k_0-1}_{l=j} E_l\\
&\leq 2^{-n}3\epsilon\epsilon_0  M.
\end{align*}
Moreover, by \eqref{eq0410fri4} with $\epsilon_0=2^{-n-2}$, we have
\[
\fint_{2B_j} \abs{Du}\,dx\leq 2^{-n-1} M.
\]
Therefore, we obtain that
\[
\fint_{2B_{k_0}}\abs{Du} \,dx\leq \fint_{2B_{k_0}}\abs{Du  -\overline{Du}^{x_0,r_j}}\, dx+2^n \fint_{2B_j}\abs{Du} \,dx\leq   2^{-n}3\epsilon\epsilon_0  M+  2^{-1} M\leq M.
\]
This proves the third claim in \eqref{eq0412sun4} for $k=k_0$. Thus, we complete the proof.
\end{proof}

Now, we are ready to obtain the local boundedness of $Du$.

\begin{proposition} \label{prop0413mon1}
Let $u\in \mathcal{A}_\psi$ be the solution to the variational inequality \eqref{main_vari ineq}. Suppose the mean oscillation of the $\mathsf{data}$ satisfies the Dini condition with constant $\kappa>0$. 
Then $Du$ is locally bounded in $\Omega$. Moreover,  for any Lebesgue point $x_0\in\Omega$ of $|Du|$ and $B_{2r}(x_0)\Subset \Omega$ with $r\le \min\{R_0,R_1\}$, where $R_1>0$ is determined by the choice $\epsilon=\epsilon_0=2^{-n-2}$,  we have
\begin{equation}\label{eq0413mon1}
|Du(x_0)|\leq C\bigg( \fint_{B_{2r}(x_0)} \abs{Du} \,dx  +   \norm{F}_{L^\infty(B_{2r}(x_0))}
^{\frac{1}{p-1}}+\norm{D\psi}_{L^\infty(B_{2r}(x_0))} +1   \bigg)
\end{equation}
where $C=C(n,p,\Lambda,\omega_\mathsf{data},\kappa)\geq 1$.
\end{proposition}

\begin{remark}
By the density of Lebesgue point and a standard normalization argument, we can derive the following Lipschitz estimate
\[
\norm{Du}_{L^\infty(B_r(x_0))}\leq C\bigg( \fint_{B_{2r}(x_0)} \abs{Du} \,dx  +   \norm{F}_{L^\infty(B_{2r}(x_0))}
^{\frac{1}{p-1}}+\norm{D\psi}_{L^\infty(B_{2r}(x_0))}   \bigg)
\]
for any $B_{2r}(x_0)\Subset\Omega$ with  sufficiently small $r>0$.
\end{remark}

\begin{proof}
We recall the setting in the beginning of the section when $\epsilon=\epsilon_0$, where $x_0\in \Omega$  is a Lebesgue point of $\abs{Du}$. 
We then define 
\[
M:=\frac{\delta^{-2n}}{3\epsilon_0^3}\bigg(\fint_{B_{2r}(x_0)} \abs{Du} \,dx  +   \norm{F}_{L^\infty(B_{2r}(x_0))}
^{\frac{1}{p-1}}+\norm{D\psi}_{L^\infty(B_{2r}(x_0))}+1   \bigg),
\]
\begin{equation*}
F_l:=\frac{1}{\delta^n \epsilon_0} E_l+\fint_{2B_l} \abs{Du} \,dx
\quad \text{for }\ l\in \mathbb{N}. 
\end{equation*} 
Note that
\begin{align*}
F_1\leq\bigg(  \frac{2^n}{\epsilon_0\delta^{2n}}+\frac{1}{\delta^{n}} \bigg)\bigg(\fint_{B_{2r}(x_0)} \abs{Du}\, dx \bigg)
&\leq\frac{\delta^{-2n}}{\epsilon_0^2}\bigg(  \frac{1}{ 4}+\epsilon_0^2 \delta^n \bigg)\bigg(\fint_{B_{2r}(x_0)} \abs{Du} \,dx \bigg)\\
&\leq\frac{\delta^{-2n}}{2\epsilon_0^2} \bigg(\fint_{B_{2r}(x_0)} \abs{Du} \,dx \bigg)\leq \frac{3}{2}\epsilon_0 M.
\end{align*}
Then, we can assume that there exists $j_0\in \mathbb{N}$ such that
\begin{equation}\label{eq0413mon2}
F_{j_0}\leq \frac{3}{2}\epsilon_0 M\qquad\hbox{and}\qquad F_{l}>\frac{3}{2}\epsilon_0 M\;~\;\hbox{for all $l>j_0$}.
\end{equation}
Because, if it is not true, one can find a sequence $\set{j_l}_{l\in\mathbb{N}}$ with $j_1<j_2<\ldots$ such that
\[
F_{l_j}\leq \frac{3}{2}\epsilon_0 M\quad\hbox{for all $j\in \mathbb{N}$},
\]
which together with $x_0$ is the Lebesgue point of $Du$, it implies that
\begin{equation}\label{eq0413mon3}
\abs{Du(x_0)}=\lim_{l\rightarrow \infty} \abs{\overline{Du}^{x_0,r_j}}\leq \liminf_{j\rightarrow \infty} F_{l_j}\leq \frac{3}{2}\epsilon_0 M.
\end{equation}
This prove \eqref{eq0413mon1}. Thus, we now assume \eqref{eq0413mon2}. Then, for every $l> j_0$,
\[
F_l >\frac{3}{2}\epsilon_0 M 
\]
From Lemma \ref{lemma0413mon1} with $j=j_0$ and $k=j_0+1$, 
we have that
\[
E_{j_0+1} \leq \frac{\delta^n\epsilon_0^2}{2}M.
\]
It implies that
\[
\fint_{2B_{j_0+1}}\abs{Du}\, dx= F_{j_0+1} -\frac{1}{\delta^n \epsilon_0} E_{j_0+1}\geq  \frac{3}{2}\epsilon_0 M -\frac{1}{2}\epsilon_0 M\geq \epsilon_0 M.
\]
From this, we conclude that \eqref{eq0410fri5} is satisfied for $l=j_0+1$. Then by Lemma \ref{lemma0413mon1} again with $k=j_0+1$, which implies that
\[
E_{j_0+2} \leq \frac{\delta^n\epsilon_0^2}{2}M.
\]
By repeating the same process, \eqref{eq0410fri5} holds for every $l>j_0$ and so
\[
\fint_{2B_k}\abs{Du} \,dx\leq M\quad\hbox{for every $k> j_0$}.
\]
Therefore, similar to \eqref{eq0413mon3}, we prove that $\abs{Du(x_0)}\leq M$. This completes the proof.
\end{proof}

\subsection{Proof of Theorem \ref{main_theorem}}

Fix any $\Omega''\Subset\Omega'\Subset \Omega$. From Proposition \ref{prop0413mon1}, we see that 
$Du\in L^\infty(\Omega')$. 
Define 
\[
M:=\frac{2^{n+1}}{\epsilon_0}\big( \norm{Du}_{L^\infty(\Omega')}  +   \norm{F}_{L^\infty(\Omega')}
^{\frac{1}{p-1}}+\norm{D\psi}_{L^\infty(\Omega')} +1   \big).
\]
For $\epsilon\in (0,\epsilon_0]$, we choose $\delta=\delta(\epsilon)$ as in \eqref{eq:delta1} and $\hat R_{1}=\hat R_{1}(\epsilon)$ such that for every $r\in(0,\hat R_1]$,
\begin{equation}\label{eq:tR11}
(2\delta)^{-n} C_2 \omega_{\mathsf{data}}(r)^{\frac{\min\set{2,p}}{p}}\le \frac{\delta^{2n} \epsilon}{2}, 
\qquad 4\delta^{-n}C_2 \omega_{\mathsf{data}}(r)^{\frac{\min\set{2,p}}{p}}\le  \frac{\delta^{3n} \epsilon \epsilon_0 }{16},
\end{equation}
\begin{equation}\label{eq:tR12}
2^{2p}C_3\delta^{-2np}  \epsilon^{-(p-1)}\omega_{\mathsf{data}}(r)   \frac{\epsilon_0}{2^{n+1}} \le \frac{\delta^n\epsilon \epsilon_0}{2},
\end{equation}
and \eqref{eq:R12} is satisfied.
We notice from  \eqref{eq:tR11} that the conditions \eqref{eq:R11} is satisfied with $\epsilon$ replaced by  $\delta^{2n}\epsilon$ respectively, for every $r\in(0,\hat R_1]$.

Let $x_0\in \Omega''$ be a Lebesgue point of $Du$ and $B_{2r}(x_0)\Subset\Omega'$ with $r\le \min\{R_0,\hat R_1\}$. Then we immediately obtain
\begin{equation}\label{eq:mainproof1}
\fint_{B_{2r}(x_0)}\abs{Du} \, dx +  \norm{F}_{L^\infty(B_{2r}(x_0))}
^{\frac{1}{p-1}}+\norm{D\psi}_{L^\infty(B_{2r}(x_0))} +1\leq \frac{\epsilon_0}{2^{n+1}} M,
\end{equation}
\begin{equation}\label{eq:mainproof2}
\fint_{B_{\rho}(x_0)}\abs{Du} \, dx \le \frac{\epsilon_0}{2^{n+1}} M 
\quad \text{for all } \rho\in (0, 2r].
\end{equation}
Note that if 
\[
\fint_{B_{\delta^{2}r}(x_0)} |Du|\,dx \le   \frac{\delta^n\epsilon\epsilon_0}{2} M,
\]
then
\[
\fint_{B_{\delta^{2}r}(x_0)} |Du-\overline{Du}^{x_0,\delta^2r}|\,dx \le  \delta^n \epsilon\epsilon_0 M.
\]
On the other hand, suppose that
\[
\fint_{B_{\delta^{2}r}(x_0)} |Du|\,dx \ge    \frac{\delta^n\epsilon\epsilon_0}{2} M,
\]
which implies
\[
\fint_{B_{2\delta r}(x_0)} |Du|\,dx  \ge \bigg(\frac{\delta}{2}\bigg)^n \fint_{B_{\delta^2 r}(x_0)} |Du|\,dx \ge  \frac{\delta^{2n}  \epsilon \epsilon_0  }{2^{n+1}} M,
\]
Then, by Lemma \ref{lemma0331tue1} with $j=0$ and with $M$ and $\epsilon$ replaced by $\frac{\epsilon_0}{2^{n+1}} M$ and $\delta^{2n}  \epsilon$ respectively, \eqref{eq:mainproof2}, and \eqref{eq:tR12},  we obtain
\[\begin{split}
&\fint_{B_{\delta^2 r}(x_0)} |Du-\overline{Du}^{x_0,\delta^2 r}|\,dx \\
& \leq \frac{\delta^{2n} \epsilon}{32} \fint_{B_{\delta r}(x_0)} |Du-\overline{Du}^{x_0,\delta r}|\,dx +2^{2p}C_3\delta^{-2n} ( \delta^{2n} \epsilon)^{-(p-1)}\omega_{\mathsf{data}}(\delta^j r)   \frac{\epsilon_0}{2^{n+1}} M \\
& \leq \frac{ \delta^{2n} \epsilon \epsilon_0}{2^{n+5}} M+ \frac{ \delta^{n} \epsilon \epsilon_0}{2}  M \le   \delta^{n} \epsilon \epsilon_0  M.
\end{split}
\]
Therefore, since $r\le \min\{R_0,\tilde R_1\}$ is arbitrary, we have proved that
\begin{equation}\label{eq:mainproof3}\begin{split}
&\fint_{B_\rho(x_0)}|Du-\overline{Du}^{x_0,\rho}|\, dx  \le \delta^{n} \epsilon \epsilon_0  M \\
&\qquad \text{for all Lebesgue points $x_0\in \Omega''$ of $Du$ and all $\rho\in(0, R_2]$,}
\end{split}\end{equation}
where $R_2=R_2(\epsilon):= \frac{\delta^2}{2} \min\{R_0,\hat R_1, \mathrm{dist}(\Omega'',\partial \Omega')\}$.

We recall the notation in the beginning of the section with $r=R_2$ and any Lebesgue point $x_0\in \Omega$. Note that $r$ depends on $\epsilon$.  
Then we claim that
\[
\textbf{(Claim)}\qquad \bigabs{ \overline{Du}^{x_0 ,r_2}  -\overline{Du}^{x_0 ,r_{k+1}}} \leq (2+2^{n+1})\epsilon M 
\quad \text{for all $\ k> 3$.}
\]
If the claim holds, then $Du$ can be represented by a uniformly continuous function in $\Omega''$. Indeed,  if $x, y\in \Omega''$ be any Lebesgue points of $Du$ with $\abs{x-y}$ sufficiently small so that $\abs{B_{\delta^2r}(x)\backslash B_{\delta^2r}(y)}\leq \epsilon\abs{B_{\delta^2r}}$, then
\begin{align*}
&\abs{Du(x)-Du(y)} \\
&\leq \limsup_{k\rightarrow\infty}\bigabs{ \overline{Du}^{x ,r_2}  -\overline{Du}^{x ,r_{k+1}}} +\limsup_{k\rightarrow\infty}\bigabs{ \overline{Du}^{y ,r_2}  -\overline{Du}^{y ,r_{k+1}}} + \bigabs{ \overline{Du}^{x ,r_2}  -\overline{Du}^{y ,r_{2}}}\\
&\leq 2 (2+2^{n+1})\epsilon M+\frac{1}{\abs{B_{\delta^2r}}} \big(\abs{B_{\delta^2r}(x)\backslash B_{\delta^2r}(y)} +\abs{B_{\delta^2r}(y)\backslash B_{\delta^2r}(x)} \big)\norm{Du}_{L^\infty(\Omega')} \\
&\leq 2^{n+3} \epsilon M.
\end{align*}
Therefore, we conclude that $Du$ has a uniformly continuous representative in $\Omega''$. We complete the proof verifying the claim.

\begin{proof}[Proof of the claim]
We notice  from \eqref{eq:mainproof1} and \eqref{eq:mainproof3} that the first assumption \eqref{eq0410fri4} in Lemma \ref{lemma0413mon1}  holds for all $j\ge 2$ and $E_l \le \delta^n \epsilon \epsilon_0 M$ for all $l\ge 2$. Set
\[
\mathcal{I}:= \bigset{i\in \mathbb{N}~:~ F_i<2\epsilon M},
\quad F_{i}:= \frac{1}{\delta^n \epsilon_0} E_i+\fint_{2B_i} \abs{Du} \,dx.
\]
We first suppose that $\mathcal{I}\cap \set{  3,\ldots,k-1}=\emptyset$. Then we have
\begin{equation}\label{eq0414tue1}
\fint_{2B_l}\abs{Du} \, dx= F_l-\frac{1}{\delta^n\epsilon_0} E_l\geq 2\epsilon M-\epsilon M\geq \epsilon M\quad\hbox{for all $\ 3\leq l\leq k-1$}.
\end{equation}
Thus, the second assumption \eqref{eq0410fri5} in Lemma \ref{lemma0413mon1} holds when $j=2$, so we have
\[
\bigabs{ \overline{Du}^{x_0 ,r_2}  -\overline{Du}^{x_0 ,r_{k+1}}}\leq 3\epsilon\epsilon_0  M\leq (2+2^{n+1})\epsilon M
\]
We next suppose that $\mathcal{I}\cap \set{  3,\ldots,k-1}\neq\emptyset$. Let  $k' =\min \set {\mathcal{I}\cap \set{  3,\ldots,k-1}}$ and $j' =\max \set {\mathcal{I}\cap\set{ 3,\ldots,k-1}}$. Now we consider the following four cases:

(i) $k'=3 $:  since $3\in \mathcal{I}$,
\begin{align*}
\bigabs{ \overline{Du}^{x_0 ,r_{2}} }= \bigabs{ \overline{Du}^{x_0 ,r_{2}} -\overline{Du}^{x_0 ,r_{3}} }+\bigabs{ \overline{Du}^{x_0 ,r_{3}} }\leq \delta^{- n} E_2+ 2^{n}\fint_{2B_3} \abs{Du} \,dx
& \leq \epsilon_0 \epsilon M+2^{n+1}\epsilon M\\
& \le (1+2^{n+1})\epsilon M.
\end{align*}

(ii) $k'>3$:  this means that $\mathcal{I}\cap \set{  3,\ldots,k'-1}=\emptyset$.  Thus, by \eqref{eq0414tue1} when $k=k'-1> 3$, we can apply Lemma \ref{lemma0413mon1} to $j=2$ and $k=k'-1$, so that 
\[
\bigabs{ \overline{Du}^{x_0 ,r_2}  -\overline{Du}^{x_0 ,r_{k'}}}\leq 3\epsilon\epsilon_0  M \le \epsilon M.
\]
Hence, since $k'\in \mathcal{I}$, we obtain
\[
\bigabs{ \overline{Du}^{x_0 ,r_2}  }\leq \bigabs{ \overline{Du}^{x_0 ,r_2}  -\overline{Du}^{x_0 ,r_{k'}}}+\bigabs{  \overline{Du}^{x_0 ,r_{k'}}}\leq 3\epsilon\epsilon_0  M+ 2^n F_{k'}\leq (1+2^{n+1})\epsilon M.
\]

(iii) $j'=k-1$: Since $k-1\in\mathcal{I}$
\begin{align*}
\bigabs{ \overline{Du}^{x_0 ,r_{k+1}} }\leq \delta^{-n}\big( E_k+E_{k-1}\big)+ 2^{n}\fint_{2B_{k-1}} \abs{Du} \,dx \leq 2\epsilon_0 \epsilon M+2^{n+1}\epsilon M \le (1+2^{n+1})\epsilon M
\end{align*}

(iv) $j' <k-1$: this implies that $\mathcal{I}\cap \set{  j'+1,\ldots,k-1}=\emptyset$ and \eqref{eq0414tue1}. Thus, by \eqref{eq0414tue1}, we can apply Lemma \ref{lemma0413mon1} to $j=j'$, we obtain
\[
\bigabs{ \overline{Du}^{x_0 ,r_{j'}}  -\overline{Du}^{x_0 ,r_{k+1} }}\leq 3\epsilon\epsilon_0  M \le \epsilon M
\]
Hence, since $j'\in \mathcal{I}$, we obtain

\[
\bigabs{ \overline{Du}^{x_0 ,r_{k+1}}  }\leq \bigabs{ \overline{Du}^{x_0 ,r_{j'}}  -\overline{Du}^{x_0 ,r_{k+1}}}+\bigabs{  \overline{Du}^{x_0 ,r_{j'}}}\leq \epsilon  M+ 2^n \fint_{2B_{j'}} \abs{Du} \,dx \leq (1+2^{n+1})\epsilon M.
\]

Therefore, by combining all of the cases, we have
\[
 \bigabs{ \overline{Du}^{x_0 ,r_{2}}  -\overline{Du}^{x_0 ,r_{k+1}}}\leq \bigabs{  \overline{Du}^{x_0 ,r_{2}}}+\bigabs{  \overline{Du}^{x_0 ,r_{k+1}}}\leq   (2+2^{n+1})\epsilon M.
\]
This completes the claim.
\end{proof}


\end{document}